\documentclass[11pt,a4paper]{article}
\usepackage{mathrsfs}
\usepackage{amsmath,amsthm,color,graphicx,eepic}
\usepackage{amssymb} \usepackage{verbatim}
\usepackage{pict2e}

\usepackage{hyperref}

\newtheorem{theorem}{Theorem}

\newtheorem{lemma}{Lemma}

\newtheorem{claim}{Claim}
\newtheorem{subclaim}{Claim}[claim]

\newcommand{\ex}{{\rm ex}}
\newcommand{\en}{{\rm end}}

\begin{document}

\title{ On graphs without cycles of length $0$ modulo $4$}

\author{
Ervin Gy\H{o}ri\thanks{Alfr\'{e}d R\'{e}nyi Institute of Mathematics, Budapest, Hungary (\texttt{gyori.ervin@renyi.hu}, \texttt{tompkins.casey@renyi.hu}).} \and \hspace{-1em}
Binlong Li\thanks{School of Mathematics and Statistics,
Northwestern Polytechnical University, Xi’an, 710072, China.} \thanks{Xi'an-Budapest Joint Research Center for Combinatorics, Northwestern Polytechnical University, Xi'an, 710072,  China (\texttt{binlongli@nwpu.edu.cn}).}
\and \hspace{-1em}
Nika Salia\thanks{King Fahd University of Petroleum and Minerals, Dhahran, Saudi Arabia (\texttt{nikasalia@yahoo.com}).}
\and \hspace{-1em}
Casey Tompkins\footnotemark[1]
\and \hspace{-1em} Kitti Varga\thanks{Budapest University of Technology and Economics, Budapest, Hungary.} \thanks{MTA-ELTE Egerv\'ary Research Group, Budapest, Hungary (\texttt{vkitti@renyi.hu}).}
\and \hspace{-1em}
Manran Zhu\thanks{Corvinus University, Center for Collective Learning, Budapest, Hungary (\texttt{manran.zhu@uni-corvinus.hu}).}
}

\date{}
\maketitle

\begin{abstract}
Bollob\'as proved that for every $k$ and $\ell$ such that $k\mathbb{Z}+\ell$ contains an even number, an $n$-vertex graph containing no cycle of length $\ell \bmod k$ can contain at most a linear number of edges.  
The precise (or asymptotic) value of the maximum number of edges in such a graph is known for very few pairs $\ell$ and $k$. 
In this work we precisely determine the maximum number of edges in a graph containing no cycle of length $0 \bmod 4$.
\end{abstract}

\section{Introduction}
It is well-known that $n$-vertex graphs containing no even cycles can contain at most $\lfloor \frac{3}{2}(n-1) \rfloor$ edges.  
On the other hand, if only a set of odd cycles are forbidden, then taking a balanced complete bipartite graph yields $\lfloor \frac{n^2}{4} \rfloor$ edges, and this is sharp for sufficiently large $n$~\cite{simonovits1974extremal}. 
Given these observations it was natural to consider the extremal problem where for natural numbers $k$ and $\ell$ such that $k\mathbb{Z}+\ell$ contains an even number, all cycles of length $\ell \bmod k$ are forbidden. 
It was conjectured by Burr and Erd\H{o}s~\cite{erdos} that such a graph could contain at most a linear number of edges. 
This conjecture was proved by Bollob\'as~\cite{bollobas1977cycles}. 

Given the result of Bollob\'as, it is interesting to determine the smallest constant $c_{\ell,k}$ (where $k\mathbb{Z}+\ell$ contains an even number) such that every $n$-vertex graph with $c_{\ell,k} n$ edges must contain a cycle of length~$\ell \bmod k$. 
The problem of finding such an optimal $c_{\ell,k}$ was mentioned by Erd\H{o}s in~\cite{erdos1995some}. 
Various improvements to the general bounds on $c_{\ell,k}$ have been obtained~\cite{thomassen1983graph,thomassen1986paths,verstraete2000arithmetic}
culminating in a recent result of Sudakov and Verstra\"ete~\cite{sudakov} showing that for $3\le \ell < k$, the value of $c_{\ell,k}$ is within an absolute constant of the maximum number of edges in a $k$-vertex $C_\ell$-free graph. Thus, for even $\ell\ge 4$ the general problem of determining $c_{\ell,k}$ is at least as difficult as determining the Tur\'an number of $C_\ell$ (for which we only know the order of magnitude when $\ell\in \{4,6,10\}$). 

The precise value of $c_{\ell,k}$ is known for very few pairs $\ell$ and $k$.  
As mentioned above it is well-known that $c_{0,2}=\frac{3}{2}$. 
It was proved that $c_{0,3}=2$ by Chen and Saito~\cite{chen1994graphs}, which resolved a conjecture of Barefoot~et~al~\cite{barefoot1991cycles}.
The $n$-vertex graph avoiding all cycles of length $0\bmod 3$ with the maximum number of edges is the complete bipartite graph~$K_{2,n-2}$. 
In fact Chen and Saito~\cite{chen1994graphs} proved a stronger result (also conjectured by Barefoot~et~al~\cite{barefoot1991cycles}) that a graph of minimum degree at least~$3$ contains a cycle of length $0\bmod 3$, which implies the aforementioned results. 

 Dean, Kaneko, Ota and Toft~\cite{Deanetal} (see also Saito~\cite{saito1992cycles}) showed that every $n$-vertex $2$-connected graph of minimum degree at least~$3$ either contains a cycle of length $2 \bmod 3$ or is isomorphic to $K_4$ or $K_{3,n-3}$. 
 From this result it is easily deduced that for $n$ sufficiently large, $K_{3,n-3}$ maximizes the number of edges in a graph not containing a cycle of length $2 \bmod 3$. Consequently, $c_{2,3}=3$.  
 
 The situation for cycles of length $1 \bmod 3$ is less clear.  
 Dean, Kaneko, Ota and Toft~\cite{Deanetal} proved that every $2$-connected graph of minimum degree at least~$3$ and no cycle of length~$1 \bmod 3$ contains a Petersen graph as a subgraph. 
 This result was strengthened by Mei and Zhengguang~\cite{mei2001cycles} who showed that in fact every such graph contains a  Petersen graph as an induced subgraph. 
 However, it is not clear how one could derive a result on the maximum number of edges from these results. 
 Thus, determining $c_{1,3}$ remains open. 
 A general estimate of $c_{\ell,3} \le \ell+2$ was given in the original paper of Erd\H{o}s~\cite{erdos}.

Gao, Li, Ma and Xie~\cite{gao2022two} proved that an $n$-vertex graph $G$ with at least $\frac{5}{2}(n-1)$ edges contains two consecutive even cycles unless $4\mid (n-1)$ and every block of $G$ is isomorphic to $K_5$. 
This result settled the $k=2$ case of conjecture of Verstra\"ete~\cite{verstraete2016extremal} about the maximum number of edges in graphs avoiding cycles of $k$ consecutive lengths. 
As a consequence of this result Gao, Li, Ma and Xie proved that $c_{2,4}=\frac{5}{2}$.

 In the present paper we will consider the problem of maximizing edges in a graph containing no cycle of length~$0 \bmod 4$.
 This is the last remaining class modulo~$4$ since the others contain only odd numbers. 
 An extensive investigation of such graphs was undertaken by Dean, Lesniak and Saito~\cite{dean1993cycles}.  
 They proved, among several other results, that $c_{0,4}\le 2$.  
 
 Our main result is an exact determination of~$c_{0,4}$. 
 In fact we determine a sharp upper bound on the number of edges in a graph containing no cycle of length $0 \bmod 4$, and as a consequence we obtain $c_{0,4} =\frac{19}{12}$.

\begin{theorem}\label{ThMain}
Let $G$ be an $n$-vertex graph. If $e(G)>\lfloor\frac{19}{12}(n-1)\rfloor$, then $G$ contains a cycle of length $0\bmod 4$.
\end{theorem}
Constructions attaining this upper bound for every $n\ge 2$ are given in Section~\ref{extremal}.

\section{Some preliminaries}

Let $G$ be a graph and $x,y\in V(G)$. A path from $x$ to $y$ is called an \emph{$(x,y)$-path}. If $X,Y$ are two subgraphs of $G$ or subsets of $V(G)$, then a path from $X$ to $Y$ is an $(x,y)$-path with $x\in X$, $y\in Y$, and all internal vertices in $V(G)\backslash(X\cup Y)$. A path (cycle) is \emph{even} (\emph{odd}) if its length is even (odd). The graph consisting of an odd cycle $C$, a path $P_1$ from $x$ to $C$ and a path $P_2$ from $C$ to $y$ with $V(P_1)\cap V(P_2)=\emptyset$ (not excluding the case that $P_1$ and/or $P_2$ are trivial), is called an \emph{adjustable path} from $x$ to $y$ (or briefly, an \emph{adjustable $(x,y)$-path}). Notice that an adjustable $(x,y)$-path contains both an even $(x,y)$-path and an odd $(x,y)$-path. For a path $P$ or a cycle $C$, we denote by $|P|$ or $|C|$ its length. We write $\en(P)=\{x,y\}$ if $P$ is a path or adjustable path from $x$ to $y$.

Denote by $\varTheta$ a graph consisting of three internally-disjoint paths from a vertex $x$ to a vertex $y$, and denote by $\varTheta^e$ such a graph where all three paths are even.
For $k=3,4$, define $H_k^o$ (respectively $H_k^e$) to be a subdivision of $K_4$ such that each edge of some $k$-cycle in the $K_4$ corresponds to an odd path (respectively, even path). Define the \emph{odd necklace} $N^o$ to be a graph consisting of an adjustable $(x_1,x_2)$-path $R_1$, an adjustable $(x_2,x_3)$-path $R_2$, an adjustable $(x_3,x_1)$-path $R_3$, such that $R_1,R_2,R_3$ are pairwise internally-disjoint.

\begin{lemma}\label{LeThetaNH}
  Each of $\varTheta^e$, $N^o$, $H_3^e$, $H_4^o$, $H_4^e$ contains a $(0\bmod 4)$-cycle.
\end{lemma}

\begin{proof}
  For $\varTheta^e$, let $P_1,P_2,P_3$ be three internally-disjoint even paths from $x$ to $y$. If $\varTheta^e$ contains no $(0\bmod 4)$-cycle, then $|P_1|+|P_2|\equiv|P_1|+|P_3|\equiv|P_2|+|P_3|\equiv 2\bmod 4$. Thus $2(|P_1|+|P_2|+|P_3|)\equiv 2\bmod 4$, a contradiction.

  For $N^o$, let $R_i$, $i=1,2,3$, be adjustable $(x_i,x_{i+1})$-paths (the subscripts are taken modulo 3) such that $R_1,R_2,R_3$ are pairwise internally-disjoint. 
  Thus $R_i$ contains an even $(x_i,x_{i+1})$-path and an odd $(x_i,x_{i+1})$-path. 
  It follows that there is an integer $a_i$ such that $R_i$ contains two $(x_i,x_{i+1})$-paths of length $a_i\bmod 4$ and of length $(a_i+1)\bmod 4$, respectively. Thus $N^o$ contains four cycles of lengths $\sum_{i=1}^3a_i, (\sum_{i=1}^3a_i+1), (\sum_{i=1}^3a_i+2), (\sum_{i=1}^3a_i+3)\bmod 4$, respectively, one of which is a $(0\bmod 4)$-cycle.

  For $H_3^e$, let $x_1,\ldots,x_4$ be the four vertices of $K_4$, and $P_{ij}$, $1\leq i<j\leq 4$, be the path corresponding to $x_ix_j$. Suppose that $P_{12},P_{13},P_{23}$ are even. Either $|P_{14}|+|P_{24}|$ or $|P_{14}|+|P_{34}|$ or $|P_{14}|+|P_{34}|$ is even. Without loss of generality we assume that $|P_{14}|+|P_{24}|$ is even. Thus $P_{12}\cup P_{13}\cup P_{23}\cup P_{14}\cup P_{24}$ is a $\varTheta^e$, which contains a $(0\bmod 4)$-cycle.

  For $H_4^o$ and $H_4^e$, the assertions were proved in \cite{dean1993cycles}.
\end{proof}

\begin{lemma}\label{LePlanar}
  Every non-planar graph contains a $(0\bmod 4)$-cycle.
\end{lemma}

\begin{proof}
  We show that every subdivision of $K_5$ or $K_{3,3}$ contains a $(0\bmod 4)$-cycle.

  \setcounter{claim}{0}
  \begin{claim}\label{ClColoredK5}
    An edge-colored $K_5$ with two colors contains a monochromatic cycle.
  \end{claim}

  \begin{proof}
    If a $K_5$ is colored by two colors, then at least 5 edges have the same color, which produce a monochromatic cycle.
  \end{proof}

  Let $H$ be a subdivision of $K_5$, where $x_1,\ldots,x_5$ are the five vertices of $K_5$, and let $P_{ij}$, $1\leq i<j\leq 5$, be the path of $H$ corresponding to $x_ix_j$. By Claim \ref{ClColoredK5}, there is a cycle $C$ of $K_5$ such that all edges of $C$ correspond to even paths in $H$ or correspond to odd paths in $H$.

  First suppose that all edges of $C$ correspond to even paths in $H$. If $|C|=3$, say $C=x_1x_2x_3x_1$, then $P_{12}\cup P_{23}\cup P_{13}\cup P_{14}\cup P_{24}\cup P_{34}$ is an $H_3^e$. If $|C|=4$, say $C=x_1x_2x_3x_4x_1$, then $P_{12}\cup P_{23}\cup P_{34}\cup P_{14}\cup P_{13}\cup P_{24}$ is an $H_4^e$. If $|C|=5$, say $C=x_1x_2x_3x_4x_5x_1$, then $P_{12}\cup P_{23}\cup P_{34}\cup P_{45}\cup P_{15}\cup P_{13}\cup P_{24}$ is an $H_4^e$. For each of the above cases, $H$ contains a $(0\bmod 4)$-cycle by Lemma \ref{LeThetaNH}.

  Now suppose that all edges of $C$ correspond to odd paths in $H$. If $|C|=4$, say $C=x_1x_2x_3x_4x_1$, then $P_{12}\cup P_{23}\cup P_{34}\cup P_{14}\cup P_{13}\cup P_{24}$ is an $H_4^o$, which contains a $(0\bmod 4)$-cycle.

  Assume now that $|C|=3$, say $C=x_1x_2x_3x_1$. If at least 2 edges in $\{x_1x_4,x_2x_4,x_3x_4\}$ correspond to odd paths, then there is a 4-cycle all edges of which correspond to odd paths in $H$, and we are done by the analysis above. So assume without loss of generality that $x_1x_4,x_2x_4$ correspond to even paths in $H$. It follows that $P_{13}\cup P_{23}\cup P_{14}\cup P_{24}\cup P_{15}\cup P_{25}\cup P_{45}$ is an $H_3^e$, which contains a $(0\bmod 4)$-cycle.

  Finally assume that $|C|=5$, say $C=x_1x_2x_3x_4x_5x_1$. 
  If one of the edges in $\{x_1x_3,x_2x_4,x_3x_5,x_1x_4,x_2x_5\}$ corresponds to an odd path, then there is a $4$-cycle all edges of which correspond to odd paths in $H$. If all edges in $\{x_1x_3,x_2x_4,x_3x_5,x_1x_4,x_2x_5\}$ correspond to even paths, then there is a 5-cycle all edges of which correspond to even paths in $H$. In each case we are done by the analysis above.

  \begin{claim}\label{ClColoredK33}
    An edge-colored $K_{3,3}$ with two colors, say red and blue, contains either a monochromatic cycle or a cycle consisting of a red path and a blue path both of length $2$.
  \end{claim}

  \begin{proof}
    Let $X,Y$ be the bipartite sets of the $K_{3,3}$. If at least 6 edges have the same color, then they produce a monochromatic cycle. Now assume without loss of generality that 4 edges are red and 5 edges are blue. It follows that the red edges induce a forest with exactly two components $H_1,H_2$. If one component is trivial, say $V(H_1)=\{x_1\}$ with $x_1\in X$, then there is a vertex $x_2\in X\cap V(H_2)$ that is incident to two red edges, say $x_2y_1,x_2y_2$. It follows that $x_1y_1x_2y_2x_1$ is a 4-cycle with red edges $x_2y_1,x_2y_2$ and blue edges $x_1y_1,x_1y_2$, as desired. If both $H_1,H_2$ are nontrivial, then one component contains a path of length 2, say $x_1y_1,x_1y_2\in E(H_1)$. Let $x_2\in X\cap V(H_2)$. Then $x_1y_1x_2y_2x_1$ is a 4-cycle with red edges $x_1y_1,x_1y_2$ and blue edges $x_2y_1,x_2y_2$, as desired.
  \end{proof}

  Let $H$ be a subdivision of $K_{3,3}$, where $X=\{x_1,x_2,x_3\}, Y=\{y_1,y_2,y_3\}$ be the bipartite sets of the $K_{3,3}$, and let $P_{ij}$, $1\leq i,j\leq 3$, be the path of $H$ corresponding to $x_iy_j$. By Claim \ref{ClColoredK33}, there is a cycle $C$ of $K_{3,3}$ such that either all edges of $C$ correspond to even paths in $H$ or all edges correspond to odd paths in $H$, or $C$ is a 4-cycle, and two consecutive edges of $C$ correspond to even paths in $H$ and another two consecutive edges of $C$ correspond to odd paths in $H$.

  First suppose that all edges of $C$ correspond to even paths in $H$. 
  If $|C|=4$, say $C=x_1y_1x_2y_2x_1$, then $P_{11}\cup P_{12}\cup P_{21}\cup P_{22}\cup P_{13}\cup P_{23}\cup P_{31}\cup P_{32}$ is an $H_4^e$. If $|C|=6$, say $C=x_1y_1x_2y_2x_3y_3x_1$, then $P_{11}\cup P_{21}\cup P_{22}\cup P_{32}\cup P_{33}\cup P_{13}\cup P_{12}\cup P_{23}$ is an $H_4^e$. For each case, $H$ contains a $(0\bmod 4)$-cycle.

  Now suppose that all edges of $C$ correspond to odd paths in $H$. If $|C|=4$, say $C=x_1y_1x_2y_2x_1$, then $P_{11}\cup P_{12}\cup P_{21}\cup P_{22}\cup P_{13}\cup P_{23}\cup P_{31}\cup P_{32}$ is an $H_4^o$, which contains a $(0\bmod 4)$-cycle. Now assume that $|C|=6$, say $C=x_1y_1x_2y_2x_3y_3x_1$. If one of the edges in $\{x_1y_2,x_2y_3,x_3y_1\}$ corresponds to an odd path, then there is a 4-cycle all edges of which correspond to odd paths in $H$, and we are done by the analysis above. So assume that all edges in $\{x_1y_2,x_2y_3,x_3y_1\}$ correspond to even paths. It follows that $P_{11}\cup P_{21}\cup P_{22}\cup P_{32}\cup P_{33}\cup P_{13}\cup P_{12}\cup P_{23}$ is an $H_4^e$, which contains a $(0\bmod 4)$-cycle.

  Finally suppose that $|C|=4$, say $C=x_1y_1x_2y_2x_1$, such that $P_{11},P_{12}$ are even and $P_{21},P_{22}$ are odd. It follows that $P_{11}\cup P_{12}\cup P_{21}\cup P_{22}\cup P_{13}\cup P_{31}\cup P_{32}\cup P_{33}$ is an $H_3^e$, which contains a $(0\bmod 4)$-cycle.
\end{proof}

For a path $P$ and two vertices $x,y\in V(P)$, we denote by $P[x,y]$ the subpath of $P$ with end-vertices $x$ and $y$. 
For a cycle $C$ with a given orientation and two vertices $x,y\in V(C)$, we use $C[x,y]$ (or $\overleftarrow{C}[y,x]$) to denote the path in $C$ from $x$ to $y$ along the given orientation, and $C(x,y)$ (or $P(x,y)$) is the path obtained from $C[x,y]$ (or $P[x,y]$) by removing its two end-vertices $x,y$. 

A path or adjustable path $P$ is called a \emph{bridge} of a cycle $C$ if $P$ is nontrivial, $P$ and $C$ are edge-disjoint and $V(P)\cap V(C)=\en(P)$. We remark that an adjustable bridge of $C$ contains both an even bridge and an odd bridge.
Let $P$ be a bridge of $C$, say with $\en(P)=\{x,y\}$. The \emph{span} of $P$ on $C$, denoted by $\sigma_C(P)$, is defined as $\min\{|C[x,y]|,|C[y,x]|\}$. Two bridges $P_1,P_2$ of $C$, where $\en(P_i)=\{x_i,y_i\}$, $i=1,2$, are \emph{crossed} on $C$ if $P_1,P_2$ are vertex-disjoint and $x_1,x_2,y_1,y_2$ appear in this order along $C$.

\begin{lemma}\label{LeBridge}
  Let $C$ be an even cycle and $P_i$, $i=1,2,3$, be even bridges of $C$.\\
  (1) If $P_1$ has an even span, then $C\cup P_1$ contains a $(0\bmod 4)$-cycle.\\
  (2) If $P_1,P_2$ are crossed on $C$, then $C\cup P_1\cup P_2$ contains a $(0\bmod 4)$-cycle.\\
  (3) If $P_1,P_2,P_3$ are pairwise internally-disjoint, then $C\cup P_1\cup P_2\cup P_3$ contains a $(0\bmod 4)$-cycle.
\end{lemma}

\begin{proof}
  Suppose that $\en(P_i)=\{x_i,y_i\}$ for $i=1,2,3$.

  (1) Since $C$ is even and $\sigma_C(P_1)$ is even, both $C[x_1,y_1]$ and $C[y_1,x_1]$ are even. Thus $C\cup P_1$ is a $\varTheta^e$, which contains a $(0\bmod 4)$-cycle by Lemma \ref{LeThetaNH}.

  (2) Suppose that $x_1,x_2,y_1,y_2$ appear in this order along $C$. If $\sigma_C(P_1)$ or $\sigma_C(P_2)$ is even, then we are done by (1). Now suppose that both $\sigma_C(P_1),\sigma_C(P_2)$ are odd. Assume without loss of generality that $C[x_1,x_2]$ is even, which implies that $C[x_2,y_1]$ is odd, $C[y_1,y_2]$ is even and $C[y_2,x_1]$ is odd. Thus $C\cup P_1\cup P_2$ is an $H_4^e$, which contains a $(0\bmod 4)$-cycle.

  (3) By (1) we can assume that each of the bridges $P_1,P_2,P_3$ has an odd span. By (2) we can assume that no two of the bridges $P_1,P_2,P_3$ are crossed. First suppose that $x_1,y_1,x_2,y_2,x_3,y_3$ appear in this order along $C$ (possibly $y_1=x_2$ or $y_2=x_3$ or $y_3=x_1$). It follows that $P_i\cup C[x_i,y_i]$ is an odd cycle for $i=1,2,3$, which implies that $C\cup P_1\cup P_2\cup P_3$ is an $N^o$, and thus contains a $(0\bmod 4)$-cycle.

  Now suppose that $x_1,x_2,x_3,y_3,y_2,y_1$ appear in this order along $C$. Notice that $C[x_2,y_2]\cup P_3$ and $C[y_2,x_2]\cup P_1$ are two adjustable $(x_2,y_2)$ paths, and contain two even $(x_2,y_2)$-paths. Together with $P_2$, we obtain a $\varTheta^e$, which contains a $(0\bmod 4)$-cycle.
\end{proof}

\begin{lemma}\label{LeBridgeCrossed}
  Let $C$ be an even cycle, $P_1,P_2$ be crossed bridges of $C$, and $R$ be an adjustable path from $P_2-C$ to $C$, such that $P_1$ is even and $P_1,R$ are internally-disjoint. Then $C\cup P_1\cup P_2\cup R$ contains a $(0\bmod 4)$-cycle.
\end{lemma}

\begin{center}
\begin{picture}(145,125)\label{FiBridgeCrossed}

\put(10,15){\put(50,50){\circle{100}} \qbezier(50,0)(120,-15)(125,50) \qbezier(50,100)(120,115)(125,50) \qbezier(15,15)(50,50)(85,15)
\put(50,65){\circle{20}} \put(50,33){\line(0,1){22}} \put(50,100){\line(0,-1){25}}
\put(50,0){\circle*{4}} \put(50,100){\circle*{4}} \put(15,15){\circle*{4}} \put(85,15){\circle*{4}} \put(50,33){\circle*{4}}
\put(47,105){$x_1$} \put(47,-9){$y_1$} \put(11,22){$x_2$} \put(81,22){$y_2$} \put(42,92){$x$} \put(42,35){$y$} \put(-5,70){$C$} \put(125,60){$P_1$} \put(30,17){$P_2$} \put(52,77){$R$} }

\end{picture}

\small Figure 1. Construction of Lemma \ref{LeBridgeCrossed}.
\end{center}

\begin{proof}
  Set $\en(P_i)=\{x_i,y_i\}$, $i=1,2$, $\en(R)=\{x,y\}$, such that $x_1,x_2,y_1,y_2$ appear in this order along $C$ and $x\in V(C)$, $y\in V(P_2)\backslash\{x_2,y_2\}$ (see Figure 1). If $\sigma_C(P_1)$ is even, or $P_2$ is even, then we are done by Lemma \ref{LeBridge}. So we assume that $\sigma_C(P_1)$ is odd and $P_2$ is odd. We claim that $x=x_1$ or $y_1$. Suppose otherwise and without loss of generality that $x\in V(C(x_1,y_1))$. It follows that $R\cup P_2[y,y_2]$ is an adjustable bridge of $C$ that is crossed with $P_1$. By Lemma \ref{LeBridge}, $C\cup P_1\cup R\cup P_2[y,y_2]$ contains a $(0\bmod 4)$-cycle. Thus we conclude without loss of generality that $x=x_1$.

  If $C[x_1,x_2]$ is even, then $R\cup P_2[y,x_2]$ is an adjustable bridge of $C$ with an even span. By Lemma~\ref{LeBridge}, $C\cup R\cup P_2[y,x_2]$ contains a $(0\bmod 4)$-cycle. So we assume that $C[x_1,x_2]$ is odd, and similarly, $C[y_2,x_1]$ is odd, from which it follows that $C[x_2,y_1]$ and $C[y_1,y_2]$ are even. Recall that $P_2$ is odd, implying that either $P_2[y,x_2]$ or $P_2[y,y_2]$ is odd. Without loss of generality we assume that $P_2[y,x_2]$ is odd. Then $C[x_1,x_2]x_2P_2y$ and $P_1y_1C[y_1,y_2]y_2P_2[y_2,y]$ are two even $(x,y)$-paths, and together with an even $(x,y)$-path in $R$ we obtain a $\varTheta^e$, which contains a $(0\bmod 4)$-cycle, as desired.
\end{proof}

\begin{lemma}\label{LeBridgeAdjustablePath}
  Let $C$ be an even cycle, $P_1,P_2$ be two vertex-disjoint bridges of $C$ with even spans, and $R$ be an adjustable path from $P_1-C$ to $P_2-C$, such that $C$ and $R$ are vertex-disjoint. Then $C\cup P_1\cup P_2\cup R$ contains a $(0\bmod 4)$-cycle.
\end{lemma}

\begin{center}
\begin{picture}(150,125)\label{FiBridgeAdjustablePath}

\put(10,15){\put(50,50){\circle{100}} \qbezier(50,0)(120,-10)(93,25) \qbezier(50,100)(120,110)(93,75) \qbezier(100,7)(115,20)(120,40) \qbezier(100,93)(115,80)(120,60) \put(120,50){\circle{20}}
\put(50,0){\circle*{4}} \put(50,100){\circle*{4}} \put(93,25){\circle*{4}} \put(93,75){\circle*{4}} \put(100,7){\circle*{4}} \put(100,93){\circle*{4}}
\put(47,105){$y_1$} \put(80,73){$x_1$} \put(80,23){$y_2$} \put(47,-9){$x_2$} \put(103,92){$x$} \put(103,4){$y$} \put(-5,70){$C$} \put(75,104){$P_1$} \put(75,-11){$P_2$} \put(120,62){$R$} }

\end{picture}

\small Figure 2. Construction of Lemma \ref{LeBridgeAdjustablePath}.
\end{center}

\begin{proof}
  Set $\en(P_i)=\{x_i,y_i\}$, $i=1,2$, $\en(R)=\{x,y\}$, where $x\in V(P_1)\backslash\{x_1,y_1\}$, $y\in V(P_2)\backslash\{x_2,y_2\}$ (see Figure 2). 
  If $P_1,P_2$ are crossed on $C$, then $C\cup P_1\cup P_2\cup R$ contains a subdivision of $K_{3,3}$, and thus contains a $(0\bmod 4)$-cycle by Lemma~\ref{LePlanar}. 
  So we assume without loss of generality that $x_1,y_1,x_2,y_2$ appear in this order along $C$. If $P_1$ or $P_2$ is even, then we are done by Lemma \ref{LeBridge}. So we assume that both $P_1$ and $P_2$ are odd. It follows that $P_1\cup C[x_1,x_2]$ is an adjustable $(x,x_2)$-path and $P_2\cup C[x_2,y_2]$ is an adjustable $(x_2,y)$-path. Together with $R$, we get a $N^o$, which contains a $(0\bmod 4)$-cycle.
\end{proof}

\begin{lemma}\label{LeTwoCycleBrige}
  Let $C_1,C_2$ be odd cycles with $|C_1|\equiv|C_2|\bmod 4$, and $P_1,P_2,P_3$ be vertex-disjoint paths from $C_1$ to $C_2$.\\
  (1) If $C_1,C_2$ are vertex-disjoint, and $|P_1|+|P_2|$ even, then $C_1\cup C_2\cup P_1\cup P_2$ contains a $(0\bmod 4)$-cycle.\\
  (2) If $V(C_1)\cap V(C_2)=\{x\}$, $P_1$ is even and $x\notin V(P_1)$, then $C_1\cup C_2\cup P_1$ contains a $(0\bmod 4)$-cycle.\\
  (3) If $C_1,C_2$ are vertex-disjoint, then $C_1\cup C_2\cup P_1\cup P_2\cup P_3$ contains a $(0\bmod 4)$-cycle.
\end{lemma}

\begin{proof}
  Suppose that $\en(P_i)=\{x_i,y_i\}$, where $x_i\in V(C_1),y_i\in V(C_2)$ for $i=1,2,3$.

  (1) Notice that $C_1$ contains two paths from $x_1$ to $x_2$, one of which is even and the other is odd. Let $P_1^e$ and $P_1^o$, respectively, be the even and odd $(x_1,x_2)$-paths of $C_1$, and similarly let $P_2^e$ and $P_2^o$, respectively, be the even and odd $(y_1,y_2)$-paths of $C_2$. It follows that $P_1\cup P_2\cup P_1^e\cup P_2^e$ and $P_1\cup P_2\cup P_1^o\cup P_2^o$ are two even cycles. If they are not $(0\bmod 4)$-cycles, then both of them have length $2\bmod 4$. This implies that $|C_1|+|C_2|+2(|P_1|+|P_2|)\equiv 0\bmod 4$, and then $|C_1|+|C_2|\equiv 0\bmod 4$, a contradiction.

  (2) This is a degenerate case of (1), and the proof is identical to (1).

  (3) Either $|P_1|+|P_2|$, or $|P_1|+|P_3|$, or $|P_2|+|P_3|$ is even, and the assertion can be deduced from~(1).
\end{proof}

\begin{lemma}\label{LeThreeCycleBridge}
  Let $C_1,C_2,C_3$ be three odd cycles with $|C_1|\equiv|C_2|\equiv|C_3|\bmod 4$ such that they pairwise intersect at a vertex $x$. Let $P_i$ be a path from $C_i$ to $C_{i+1}$ that is vertex-disjoint with $C_{i+2}$, $i=1,2,3$ (the subscripts are taken modulo~$3$), such that $P_1,P_2,P_3$ are pairwise internally-disjoint. Then $C_1\cup C_2\cup C_3\cup P_1\cup P_2\cup P_3$ contains a $(0\bmod 4)$-cycle.
\end{lemma}

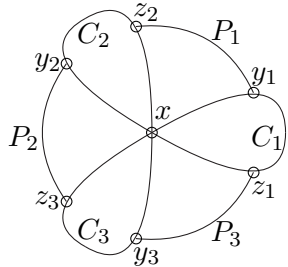
\begin{center}
\begin{picture}(120,120)\label{FiThreeCycleBridge}

\put(10,10){\qbezier(50,50)(100,79)(100,50) \qbezier(50,50)(100,21)(100,50) \qbezier(50,50)(50,108)(25,93) \qbezier(50,50)(0,79)(25,93) \qbezier(50,50)(50,-8)(25,7) \qbezier(50,50)(0,21)(25,7) \qbezier(88,65)(75,93)(44,90) \qbezier(88,35)(75,7)(44,10) \qbezier(18,76)(0,50)(18,24)
\put(50,50){\circle*{4}} \put(88,65){\circle*{4}} \put(88,35){\circle*{4}} \put(44,90){\circle*{4}} \put(18,76){\circle*{4}} \put(44,10){\circle*{4}} \put(18,24){\circle*{4}}
\put(51,55){$x$} \put(87,70){$y_1$} \put(43,94){$z_2$} \put(6,75){$y_2$} \put(6,23){$z_3$} \put(43,2){$y_3$} \put(87,27){$z_1$} \put(87,45){$C_1$} \put(22,84){$C_2$} \put(22,10){$C_3$} \put(72,85){$P_1$} \put(-4,45){$P_2$} \put(72,9){$P_3$} }

\end{picture}

\small Figure 3. Construction of Lemma \ref{LeThreeCycleBridge}.
\end{center}

\begin{proof}
  Set $\en(P_i)=\{y_i,z_{i+1}\}$, $i=1,2,3$, where $y_i,z_i\in V(C_i)\backslash\{x\}$. We suppose that $x,z_i,y_i$ appear in this order along $C_i$ (see Figure 3). If one of $P_1,P_2,P_3$ is even, then we are done by Lemma \ref{LeTwoCycleBrige}. So we assume that all of $P_1,P_2,P_3$ are odd. If $C_i[z_i,y_i]$ is even (including the case $z_i=y_i$), then $P_{i-1}z_iC_i[z_i,y_i]y_iP_i$ is an even path from $C_{i-1}-x$ to $C_{i+1}-x$, and we are done by Lemma \ref{LeTwoCycleBrige}. So we assume that $C[z_i,y_i]$ is odd for $i=1,2,3$. Now $C_1[x,y_1]\cup C_2[x,y_2]\cup C_3[x,y_3]\cup P_1\cup P_2\cup P_3$ is an $H_3^e$, which contains a $(0\bmod 4)$-cycle.
\end{proof}

\begin{lemma}\label{LeThreeCyclePath}
  Let $C_1,C_2,C_3$ be three odd cycles with $|C_1|\equiv|C_2|\equiv|C_3|\bmod 4$ such that they pairwise intersect at a vertex $x$. Let $P_i$ be a path from a vertex $y$ to $C_i-x$, $i=1,2,3$, where $y\notin V(C_1)\cup V(C_2)\cup V(C_3)$, such that $P_1,P_2,P_3$ are internally-disjoint with $C_1,C_2,C_3$ and are pairwise internally-disjoint. Then $C_1\cup C_2\cup C_3\cup P_1\cup P_2\cup P_3$ contains a $(0\bmod 4)$-cycle.
\end{lemma}

\begin{center}
\begin{picture}(120,140)\label{FiThreeCyclePath}

\put(0,10){\qbezier(50,50)(100,79)(100,50) \qbezier(50,50)(100,21)(100,50) \qbezier(50,50)(50,108)(25,93) \qbezier(50,50)(0,79)(25,93) \qbezier(50,50)(50,-8)(25,7) \qbezier(50,50)(0,21)(25,7) \put(120,50){\line(-1,0){20}} \qbezier(120,50)(100,100)(44,90) \qbezier(120,50)(100,0)(44,10)
\put(50,50){\circle*{4}} \put(120,50){\circle*{4}} \put(100,50){\circle*{4}} \put(44,90){\circle*{4}} \put(44,10){\circle*{4}}
\put(51,55){$x$} \put(122,47){$y$} \put(100,42){$z_1$} \put(43,94){$z_2$} \put(43,2){$z_3$} \put(75,54){$C_1$} \put(22,84){$C_2$} \put(22,10){$C_3$} \put(103,52){$P_1$} \put(100,80){$P_2$} \put(100,13){$P_3$} }

\end{picture}

\small Figure 4. Construction of Lemma \ref{LeThreeCyclePath}.
\end{center}

\begin{proof}
  Set $\en(P_i)=\{y,z_i\}$, where $z_i\in V(C_i)\backslash\{x\}$ (see Figure 4). Notice that either $|P_1|+|P_2|$, or $|P_1|+|P_3|$, or $|P_2|+|P_3|$ is even. Assume without loss of generality that $|P_1|+|P_2|$ is even. Then $P_1yP_2$ is an even path from $C_1-x$ to $C_2-x$. By Lemma \ref{LeTwoCycleBrige}, $C_1\cup C_2\cup P_1\cup P_2$ contains a $(0\bmod 4)$-cycle.
\end{proof}

\begin{lemma}\label{LeBipartite}
  If $G$ is a bipartite graph of order $n\geq 4$ containing no $(0\bmod 4)$-cycle, then $e(G)\leq\lfloor\frac{3}{2}(n-2)\rfloor$.
\end{lemma}

\begin{proof}
  We use induction on $n$. The assertion is trivial if $n=4$. Assume now that $n\geq 5$. If $G$ has a vertex $x$ with $d(x)\leq 1$, then by induction hypothesis, $e(G-x)\leq\lfloor\frac{3}{2}(n-3)\rfloor$, and $e(G)\leq e(G-x)+1\leq\lfloor\frac{3}{2}(n-2)\rfloor$. So assume that every vertex of $G$ has degree at least 2. If $G$ is not 2-connected, then $G$ is the union of two nontrivial graphs $G_1,G_2$ of order $n_1,n_2$, respectively, where $n_1+n_2=n+1$. If $n_i\leq 3$, then $G_i$ contains a vertex of degree at most 1 in $G$, a contradiction. So we assume that both $n_1,n_2\geq 4$. By the induction hypothesis, $e(G_i)\leq\lfloor\frac{3}{2}(n_i-2)\rfloor$, and thus $e(G)=e(G_1)+e(G_2)\leq\lfloor\frac{3}{2}(n-2)\rfloor$. So we conclude that $G$ is $2$-connected.

  By Lemma~\ref{LePlanar}, $G$ is planar. Since $G$ is bipartite and contains no $(0\bmod 4)$-cycle, every face is bounded by a cycle of length at least 6. Let $f$ be the number of faces of $G$, and $f_i$ be the number of $i$-faces of $G$. By Euler's formula,
  $$n+f=2+e(G)=2+\frac{1}{2}\sum_{i\geq 6}if_i\geq 2+3f.$$
  It follows that $f\leq\frac{n}{2}-1$ and $e(G)=n+f-2\leq\lfloor\frac{3}{2}(n-2)\rfloor$.
\end{proof}

Let $\{x,y\}$ be a cut of $G$, and $H$ be a component of $G-\{x,y\}$. The graph $G'$ obtained from $G$ by first removing all the edges between $\{x,y\}$ and $H$, and then adding the edges in $\{xz: yz\in E(G), z\in V(H)\}\cup\{yz: xz\in E(G), z\in V(H)\}$, is called a \emph{switching} of $G$ at $\{x,y\}$.

\begin{lemma}\label{LeSwitching}
  If $G$ has a $2$-cut $\{x,y\}$ and $G'$ is a switching of $G$ at $\{x,y\}$, then $e(G')=e(G)$ and $G'$ has a $(0\bmod 4)$-cycle if and only if so does $G$.
\end{lemma}

\begin{proof}
  The assertion is trivial and we omit the details.
\end{proof}

\section{Proof of Theorem \ref{ThMain}}

  We proceed by induction on the order $n$ of $G$. 
  If $n\leq 7$, then $G$ contains no $(0\bmod 4)$-cycle if and only if $G$ contains no $4$-cycle. 
  Thus the assertion can be deduced from the Tur\'an number $\ex(n,C_4)$ (see~\cite{clapham1989graphs}). 
  Assume now that $G$ is a graph of order $n\geq 8$ without a $(0\bmod 4)$-cycle. 
  By Lemmas~\ref{LeThetaNH} and~\ref{LePlanar}, $G$ is planar and contains no $\varTheta^e$, $N^o$, $H_3^e$, $H_4^o$, $H_4^e$. 
  We will first obtain some structural information about $G$ from the the following claims. We remark that by Lemma~\ref{LeSwitching}, every switching of $G$ at some $2$-cut satisfies each of the following claims as well.

  \setcounter{claim}{0}
  \begin{claim}\label{ClTwoConnected}
    $G$ is $2$-connected.
  \end{claim}

  \begin{proof}
    Suppose that $G$ is not 2-connected. Then $G$ is the union of two nontrivial graphs $G_1,G_2$, intersecting at a vertex $x$. Set $n_i=n(G_i)$, $i=1,2$, where $n_1+n_2=n+1$. By the induction hypothesis, $e(G_i)\leq\lfloor\frac{19}{12}(n_i-1)\rfloor$. Thus
    \begin{align*}
       e(G)=e(G_1)+e(G_2)\leq\left\lfloor\frac{19}{12}(n_1-1)\right\rfloor+\left\lfloor\frac{19}{12}(n_2-1)\right\rfloor\leq\left\lfloor\frac{19}{12}(n-1)\right\rfloor,
    \end{align*}
    as desired.
  \end{proof}

  For a subset $U\subseteq V(G)$, we set $\rho(U)$ to be the number of edges that are incident to a vertex in $U$.

  \begin{claim}\label{ClSubsetU}
    For every subset $U\subset V(G)$, $\rho(U)>\lfloor\frac{3}{2}|U|\rfloor$.
  \end{claim}

  \begin{proof}
    Notice that $e(G-U)=e(G)-\rho(U)$. Suppose that $\rho(U)\leq\lfloor\frac{3}{2}|U|\rfloor$. By the induction hypothesis, $e(G-U)\leq\lfloor\frac{19}{12}(n-|U|-1)\rfloor$. Thus
    \begin{align*}
       e(G)=e(G-U)+\rho(U)\leq\left\lfloor\frac{19}{12}(n-|U|-1)\right\rfloor+\left\lfloor\frac{3}{2}|U|\right\rfloor
       \leq\left\lfloor\frac{19}{12}(n-1)\right\rfloor,
    \end{align*}
    as desired.
  \end{proof}

  By Claim \ref{ClSubsetU}, we see that every two vertices of degree 2 in $G$ are nonadjacent.

  \begin{claim}\label{ClOddCycle}
    If $\{x,y\}$ is a cut and $H$ is a nontrivial component of $G-\{x,y\}$, then $G[V(H)\cup\{x,y\}]$ contains an odd cycle.
  \end{claim}

  \begin{proof}
    Set $U=V(H)$ and $G_1=G[U\cup\{x,y\}]$. Since $H$ is nontrivial, $n(G_1)\geq 4$. If $G_1$ is bipartite, then by Lemma \ref{LeBipartite},
    $$\rho(U)\leq e(G_1)\leq\left\lfloor\frac{3}{2}(n(G_1)-2)\right\rfloor=\left\lfloor\frac{3}{2}|U|\right\rfloor,$$
    contradicting Claim \ref{ClSubsetU}.
  \end{proof}

  By Claims \ref{ClTwoConnected} and \ref{ClOddCycle}, we see that if $\{x,y\}$ is a cut of $G$ and $H$ is a nontrivial component of $G-\{x,y\}$, then $G[V(H)\cup\{x,y\}]$ contains an adjustable $(x,y)$-path.

  \begin{claim}\label{ClTwoComponent}
    If $\{x,y\}$ is a cut of $G$, then $G-\{x,y\}$ has exactly two components.
  \end{claim}

  \begin{proof}
    Let $H_1,H_2,H_3$ be three components of $G-\{x,y\}$. We claim that $G[V(H_i)\cup\{x,y\}]$ contains an even $(x,y)$-path for $i=1,2,3$. If $H_i$ is trivial, say $V(H_i)=\{z\}$, then $xzy$ is an even $(x,y)$-path as desired; if $H_i$ is nontrivial, then by Claim \ref{ClOddCycle}, $G[V(H_i)\cup\{x,y\}]$ contains an adjustable $(x,y)$-path, which contains an even $(x,y)$-path. Now let $P_i$ be an even $(x,y)$-path in $G[V(H_i)\cup\{x,y\}]$, $i=1,2,3$. Then $P_1\cup P_2\cup P_3$ is a $\varTheta^e$, a contradiction.
  \end{proof}

  We call a 2-cut $\{x,y\}$ of $G$ a \emph{good} cut if for each component $H$ of $G-\{x,y\}$, $G[V(H)\cup\{x,y\}]$ contains an odd cycle. From Claim \ref{ClOddCycle}, we see that the cut $\{x,y\}$ is good if either $xy\in E(G)$ or both components of $G-\{x,y\}$ are nontrivial. We denote by $T_1(x,y)$ the triangle with two special vertices $x,y$, and by $T_2(x,y)$ a $6$-cycle with a chord of even span, such that $x,y$ are the two vertices of distance~3 (see Figure~5).

\begin{center}
\begin{picture}(70,90)\label{FiT2xy}
\thicklines \put(20,10){\put(15,0){\circle*{4}} \put(0,20){\circle*{4}} \put(0,50){\circle*{4}} \put(30,20){\circle*{4}} \put(30,50){\circle*{4}} \put(15,70){\circle*{4}}
\put(15,0){\line(3,4){15}} \put(15,0){\line(-3,4){15}} \put(0,20){\line(0,1){30}} \put(30,20){\line(0,1){30}} \put(15,70){\line(3,-4){15}} \put(15,70){\line(-3,-4){15}} \put(0,50){\line(1,0){30}}
\put(20,-5){$y$} \put(20,70){$x$} }
\end{picture}

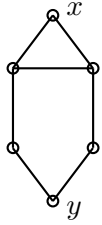
\small Figure 5. The construction of $T_2(x,y)$.
\end{center}

  \begin{claim}\label{ClGoodCut}
    Let $\{x_0,y_0\}$ be a good cut of $G$, and let $B_0,D_0$ be the two components of $G-\{x_0,y_0\}$. Let $\{x,y\}$ be a good cut of $G$ with $x,y\notin V(B_0)$ such that the component $B$ of $G-\{x,y\}$ containing $B_0$ is as large as possible. Then $G-B$ has the construction $T_1(x,y)$ or $T_2(x,y)$ (with possibly a switching at $\{x,y\}$).
  \end{claim}

  \begin{proof}
    By Claim \ref{ClTwoComponent}, $G-\{x,y\}$ has only two components $B$ and $D$. Set $G_1=G-B=G[V(D)\cup\{x,y\}]$ and $G_2=G[V(B)\cup\{x,y\}]$. Notice that $G_1$ (or $G_2$) contains an odd cycle and then contains an adjustable $(x,y)$-path.

    Suppose first that $G_1$ contains no even cycle. Then every block of $G_1$ is either a $K_2$ or an odd cycle, and at least one block of $G_1$ is an odd cycle since $\{x,y\}$ is a good cut. By the choice of $\{x,y\}$ that $B$ is maximal, we see that $G_1$ has exactly one block (which is an odd cycle). If $|G_1|\geq 5$, two adjacent vertices contained in $D$ are of degree 2, contradicting Claim \ref{ClSubsetU}. Thus $G_1$ is a triangle, which has the construction $T_1$.

    Now we assume that $G_1$ has an even cycle $C$. Let $B_1$ be the component of $G-C$ containing $B$. We choose the even cycle $C$ of $G_1$ such that $B_1$ is as large as possible. We give an orientation of $C$.

    \begin{subclaim}\label{ClOneComponent}
      $G-C$ has exactly one component $B_1$.
    \end{subclaim}

    \begin{proof}
    Suppose that $G-C$ has a second component $D_1$. We distinguish the following two cases.

    \underline{Case A. $|N_C(D_1)|\geq 3$.} Let $u_1,u_2,u_3\in N_C(D_1)$. There are three internally-disjoint paths $P_1,P_2,P_3$ from $u\in V(D_1)$ to $u_1,u_2,u_3$, respectively. Assume that $u_1,u_2,u_3$ appear in this order along $C$. We claim that $N_C(B_1)\subseteq\{u_1,u_2,u_3\}$. Suppose $B_1$ has a neighbor $v_1\in V(C)\backslash\{u_1,u_2,u_3\}$, say $v_1\in V(C(u_3,u_1))$. Notice that $C[u_1,u_3]\cup P_1\cup P_2\cup P_3$ is a $\varTheta$, and then contains an even cycle $C_1$. The component of $G-C_1$ containing $B_1$ also contains $v_1$, contradicting the choice of $C$. Thus we have that $N_C(B_1)\subseteq\{u_1,u_2,u_3\}$. It follows that $N_C(D_1)=\{u_1,u_2,u_3\}$.

    Suppose now that $N_C(B_1)=\{u_1,u_2,u_3\}$. Let $C_i=C[u_i,u_{i+1}]u_{i+1}P_{i+1}uP_iu_i$, $i=1,2,3$ (the subscripts are taken modular 3). If $C_i$ is even, then the component of $G-C_i$ containing $B_1$ also contains $u_{i+2}$, a contradiction. Thus all the three cycles $C_1,C_2,C_3$ are odd. This implies that $|C|+2(|P_1|+|P_2|+|P_3|)$ is odd, contradicting that $C$ is even. Thus we conclude that $B_1$ has exactly two neighbors on $C$, say $N_C(B_1)=\{u_1,u_3\}$. By the choice of the cut $\{x,y\}$, we see that $\{x,y\}=\{u_1,u_3\}$, say $x=u_1,y=u_3$.

    Let $R$ be an adjustable $(x,y)$-path in $G_2$, which is a bridge of $C$. If the span $\sigma_C(R)\geq 2$, then there is a bridge $P$ in $D$ from $C(y,x)$ to $C(x,y)$ (recall that $V(C)\backslash\{x,y\}$ is contained in $D$). However $P\cup P_1\cup P_2\cup P_3\cup(C-y)$ contains a $\varTheta$, and then contains an even cycle avoiding $y$, a contradiction. Thus we conclude that $\sigma_C(R)=1$, which is, $xy\in E(C)$. Since $|C|\geq 6$, either $|C[x,u_2]|\geq 3$ or $|C[u_2,y]|\geq 3$. Recall that there are no two adjacent vertices of degree 2. There is a bridge $P$ of $C$ with $\en(P)\neq\{x,y\}$. It follows that $P\cup C\cup P_1\cup P_2\cup P_3$ contains $\varTheta$ avoiding $x$ or $y$, a contradiction.

    \underline{Case B. $|N_C(D_1)|=2$.} Let $N_C(D_1)=\{u_1,u_2\}$. Note that $\{u_1,u_2\}$ is a cut of $G$. By the choice of $\{x,y\}$, we see that $D_1$ is trivial and $u_1u_2\notin E(G)$. Set $V(D_1)=\{u\}$ and $P_1=u_1uu_2$. Thus $P_1$ is an even bridge of $C$. Since $u_1u_2\notin E(G)$, we have $\sigma_C(P_1)\geq 2$. By Claim \ref{ClTwoComponent}, there is a bridge $P_2$ from $C(u_1,u_2)$ to $C(u_2,u_1)$ (in the component of $G-\{u_1,u_2\}$ not containing $u$). Set $\en(P_2)=\{v_1,v_2\}$, where $u_1,v_1,u_2,v_2$ appear in this order along $C$. Recall that $G_1$ contains an adjustable $(x,y)$-path, which can be extended to an adjustable bridge $R$ of $C$. If $\sigma_C(P_1)$ is even, or $\sigma_C(R)$ is even, then $C\cup P_1$ or $C\cup R$ contains a $(0\bmod 4)$-cycle by Lemma \ref{LeBridge}, a contradiction. So we assume that both $P_1$ and $R$ have odd spans.

    Suppose first that $P_2$ is a chord of $C$, i.e., $P_2=v_1v_2$. We claim that $N_C(B_1)\subseteq\{u_1,u_2,v_1,v_2\}$. Suppose otherwise that $B_1$ has a neighbor $v\in V(C(u_1,v_1))$. Then $C[v_1,u_1]\cup P_1\cup P_2$ is a $\varTheta$, and contains an even cycle avoiding $v$, contradicting the choice of $C$. Thus we conclude that $N_C(B_1)\subseteq\{u_1,u_2,v_1,v_2\}$, specially $\en(R)\subset\{u_1,u_2,v_1,v_2\}$. If $\en(R)=\{v_1,v_2\}$, then $R$ and $P_1$ are crossed on $C$. By Lemma \ref{LeBridge}, $C\cup P_1\cup R$ contains a $(0\bmod 4)$-cycle, a contradiction.

    Assume now that $\en(R)=\{u_1,u_2\}$. If $\sigma_C(P_2)$ is odd, then $C[v_1,v_2]v_2v_1$ is an even cycle avoiding $u_1$, contradicting the choice of $C$. So we assume that $\sigma_C(P_2)$ is even. Recall that $\sigma_C(P_1)$ is odd, implying that either $C[u_1,v_1]$ or $C[v_1,u_2]$ is even. We assume without loss of generality that $C[u_1,v_1]$ is even. It follows that $C[v_1,u_2]$ is odd, $C[u_2,v_2]$ is odd and $C[v_2,u_1]$ is even. Thus $C[u_1,v_1]\cup C[u_2,v_2]\cup P_2$ is an even $(u_1,u_2)$-path. Together wise $P_1$ and $R$, we get a $\varTheta^e$, a contradiction.

    So we conclude without loss of generality that $\en(R)=\{u_1,v_1\}$. Notice that $\sigma_C(R)$ is odd, $\sigma_C(P_1)$ is odd and $\sigma_C(P_2)$ is even. We have that $C[u_1,v_1]$ is odd, $C[v_1,u_2]$ is even, $C[u_2,v_2]$ is even and $C[v_2,u_1]$ is odd. Thus $v_1v_2C[v_2,u_1]$ and $C[v_1,u_2]u_2P_1$ are two even $(u_1,v_1)$-path. Together with $R$, we get a $\varTheta^e$, a contradiction.

    Suppose second that the internal vertices of $P_2$ are in a component $D_2$ of $G-C$ other than $B_1,D_1$. By the analysis of Case A, we see that $D_2$ is trivial as well. It follows that $P_1,P_2$ are two crossed even bridges of $C$. By Lemma \ref{LeBridge}, $C\cup P_1\cup P_2$ contains a $(0\bmod 4)$-cycle, a contradiction.

    Suppose finally that the internal vertices of $P_2$ are in $B_1$, which implies that $v_1,v_2\in N_C(B_1)$. If $\en(R)=\{v_1,v_2\}$, then by Lemma \ref{LeBridge}, $C\cup P_1\cup R$ contains a $(0\bmod 4)$-cycle, a contradiction. Thus we have that $\en(R)\neq\{v_1,v_2\}$.

    Assume now that $v_1\in\en(R)$. Recall that $R$ contains an odd cycle $C'$. Let $P'_1$ be the path in $R$ from $v_1$ to $C'$, and let $P'_2$ be a path from $v_2$ to $R-C$ with all internal vertices in $B_1$. Set $\en(P'_2)=\{v_2,z\}$. We claim that $z\in V(P'_1)\backslash\{v_1\}$. If $z\notin V(P'_1)\backslash\{v_1\}$, then $R\cup P'_2$ contains an adjustable $(v_1,v_2)$-path $R'$ (containing $C'$). If $R'$ is internally-disjoint with $C$, then by Lemma \ref{LeBridge}, $C\cup P_1\cup R'$ contains a $(0\bmod 4)$-cycle, a contradiction. So $R'$ and $C$ intersect at a third vertices which can only be contained in $C'$. It follows that there are 3 vertex-disjoint paths from $C'$ to $C$ (one of which is trivial), contradicting that $\{x,y\}$ is a cut separating $C'-\{x,y\}$ and $C-\{x,y\}$. Thus as we claimed, $z\in V(P'_1)\backslash\{v_1\}$. It follows that $P'_1[v_1,z]zP'_2$ is a bridge of $C$ which is crossed with $P_1$, and $R-(P'_1-z)$ is an adjustable path from $P'_1[v_1,z]zP'_2-C$ to $C$. By Lemma \ref{LeBridgeCrossed}, $C\cup R\cup P'_2$ contains a $(0\bmod 4)$-cycle, a contradiction.

     So we conclude that $v_1\notin\en(R)$. Let $P'_1$ be a path from $v_1$ to $R-C$ with all internal vertices in $B_1$. It follows that $R\cup P_1$ contains an adjustable path $R'$, say from $v_1$ to $z\in\en(R)$. If $R'$ is internally-disjoint with $C$, then $R'$ is an adjustable bridge of $C$ with $v_1\in\en(R')$. By the analysis above, we can get a contradiction. So assume that $R'$ and $C$ intersect at a third vertices which can only be contained in $C'$, contradicting that $\{x,y\}$ is a cut separating $C'-\{x,y\}$ and $C-\{x,y\}$.
    \end{proof}

    \begin{subclaim}\label{ClOneChord}
      $C$ has at most one chord; and if $C$ has a chord, then the chord has an even span.
    \end{subclaim}

    \begin{proof}
      Suppose that $C$ has two chords $u_1u_2$ and $v_1v_2$. Notice that $|C|\geq 6$. $C\cup\{u_1u_2,v_1v_2\}$ contains a $\varTheta$ avoiding some vertex of $C$. Thus there is an even cycle $C_1$ with $V(C_1)\subset V(C)$. It follows that the component of $G-C_1$ containing $B$ also contains $B_1$. By Claim \ref{ClOneComponent}, $G-C_1$ is connected, contradicting the choice of $C$. If $C$ has a chord $u_1u_2$ with $C[u_1,u_2]$ odd, then $C_1=u_1Cu_2u_1$ is an even cycle with $V(C_1)\subset V(C)$, also a contradiction.
    \end{proof}

    Let $V(C)=X\cup Y$ such that each two vertices in $X$ ($Y$) have an even distance on $C$.

    \begin{subclaim}\label{ClTwoNeighbor}
      $1\leq|N_X(B_1)|\leq 2$ and $1\leq|N_Y(B_1)|\leq 2$.
    \end{subclaim}

    \begin{proof}
      Suppose that $|N_X(B_1)|\geq 3$ and let $x_1,x_2,x_3\in N_X(B_1)$. There are three internally-disjoint paths $P_1,P_2,P_3$ from $u\in V(B_1)$ to $x_1,x_2,x_3$, respectively. Since each two vertices in $\{x_1,x_2,x_3\}$ have an even distance on $C$, we see that $C\cup P_1\cup P_2\cup P_3$ is an $H_3^e$, a contradiction. If $|N_X(B_1)|=0$, then there are two vertex-disjoint paths from $\{x,y\}$ to $Y$. Together with an adjustable $(x,y)$-path of $G_2$, we have an adjustable bridge $R$ of $C$ with $\sigma_C(R)$ even. By Lemma \ref{LeBridge}, $C\cup R$ contains a $(0\bmod 4)$-cycle, a contradiction. The second assertion can be proved similarly.
    \end{proof}

    \begin{subclaim}\label{ClOneNeighbor}
      Either $|N_X(B_1)|=1$ or $|N_Y(B_1)|=1$.
    \end{subclaim}

    \begin{proof}
      Suppose that $|N_X(B_1)|=2$ and $|N_Y(B_1)|=2$, say $N_X(B_1)=\{x_1,x_2\}$, $N_Y(B_1)=\{y_1,y_2\}$. It follows that $x,y\notin V(C)$. If there are two vertex-disjoint paths from $\{x,y\}$ to $\{x_1,x_2\}$ in $G-Y$, then together with an adjustable $(x,y)$-path of $G_2$, we get an adjustable bridge $R$ of $C$ with an even span. By Lemma \ref{LeBridge}, $C\cup R$ contains a $(0\bmod 4)$-cycle, a contradiction. Thus there is a vertex $x'$ separating $\{x,y\}\backslash\{x'\}$ and $X$ in $G-Y$, and similarly there is a vertex $y'$ separating $\{x,y\}\backslash\{y'\}$ and $Y$ in $G-X$, implying that $\{x',y'\}$ is a good cut of $G$. We can choose $x',y'$ such that there are two internally-disjoint paths from $x'$ to $\{x_1,x_2\}$ in $G-Y$, and there are two internally-disjoint paths from $y'$ to $\{y_1,y_2\}$ in $G-X$. By the choice of $\{x,y\}$, we see that $\{x,y\}=\{x',y'\}$, say $x=x',y=y'$.

      Let $P_1^x,P_2^x$ be two internally-disjoint paths from $x$ to $x_1,x_2$, and $P_1^y,P_2^y$ be two internally-disjoint paths from $y$ to $y_1,y_2$. Notice that $P_i^x,P_j^y$ are vertex-disjoint, $i,j=1,2$. We see that $P_1^x\cup P_2^x$ and $P_1^y\cup P_2^y$ are two bridge of $C$ with even spans. Recall that $G_2$ has an adjustable $(x,y)$-path $R$. By Lemma \ref{LeBridgeAdjustablePath}, $C\cup P_1^x\cup P_2^x\cup P_1^y\cup P_2^y\cup R$ contains a $(0\bmod 4)$-cycle, a contradiction.
    \end{proof}

    \begin{subclaim}\label{ClCycle6}
      $|C|=6$.
    \end{subclaim}

    \begin{proof}
      Suppose that $|C|\geq 10$. By Claim \ref{ClOneChord}, $C$ has at most one chord. By Claims \ref{ClTwoNeighbor} and \ref{ClOneNeighbor}, $N_C(B_1)\leq 3$. This implies that all but at most 5 vertices of $C$ have degree 2 in $G$. Since no two vertices of degree 2 are adjacent, we have that $|C|=10$, $C$ has a chord, $|N_C(B_1)|=3$, and either $N_C(B_1)\subseteq X$ or $N_C(B_1)\subseteq Y$, contradicting Claim \ref{ClTwoNeighbor}.
    \end{proof}

    Now let $C=x_1y_1x_2y_2x_3y_3x_1$, where $X=\{x_1,x_2,x_3\}, Y=\{y_1,y_2,y_3\}$.

    \begin{subclaim}\label{ClOneChordOneNeighbor}
      $C$ has a chord and $|N_X(B_1)|=|N_Y(B_1)|=1$.
    \end{subclaim}

    \begin{proof}
      By Claims \ref{ClTwoNeighbor} and \ref{ClOneNeighbor}, $|N_C(B_1)|\leq 3$. If $C$ has no chord, then $C$ contains at least three vertices of degree 2. Since no two vertices of degree 2 are adjacent, we have that $|N_C(B_1)|=3$, and either $N_C(B_1)\subseteq X$ or $N_C(B_1)\subseteq Y$, a contradiction. Thus we conclude that $C$ has a chord.

      Now suppose without loss of generality that $|N_X(B_1)|=2$ and $|N_Y(B_1)|=1$, say $N_Y(B_1)=\{y_3\}$. We claim that $y_3=x$ or $y$. Recall that there are no two vertex-disjoint paths from $\{x,y\}$ to $N_X(B_1)$ in $G-Y$. Let $x'$ be a vertex separating $\{x,y\}\backslash\{x'\}$ and $N_X(B_1)$ in $G-Y$. Then $\{x',y_3\}$ is a good cut of $G$, which implies that $\{x',y_3\}=\{x,y\}$ by the choice of $\{x,y\}$. We assume without loss of generality that $y_3=y$. By the choice of $\{x,y\}$, there are two internally-disjoint paths from $x$ to $N_X(B_1)$ not passing through $y$.

      Suppose first that $N_X(B_1)=\{x_1,x_3\}$. Notice that $\{x_1,x_3\}$ is not a good cut of $G$. This implies that $y_1y_3$ or $y_2y_3$ is the chord of $C$. However, $x_2,y_2$ or $x_2,y_1$ are two adjacent vertices of degree 2, a contradiction. So we assume without loss of generality that $N_X(B_1)=\{x_1,x_2\}$.

      Let $P_1,P_2$ be two internally-disjoint paths from $x$ to $\{x_1,x_2\}$ not passing through $y$, $R$ be an adjustable $(x,y)$-path in $G_2$. If $P_1xP_2$ is even, then it is an even bridge of $C$ with an even span. By Lemma \ref{LeBridge}, $C\cup P_1\cup P_2$ contains a $(0\bmod 4)$-cycle, a contradiction. Thus we have that $P_1xP_2$ is odd.

      Notice that $\{x_2,y_3\}$ is not a good cut of $G$. This implies that either $y_1y_2$ or $x_1x_3$ is the chord of $C$. If $x_1x_3$ is the chord, then $y_3x_1x_3y_3$ is an adjustable $(y_3,x_1)$-path, $x_1y_1x_2P_2xP_1x_1$ is an adjustable $(x_1,x)$-path. Together with the adjustable $(x,y)$-path $R$, we find an $N^o$ in $C\cup P_1\cup P_2\cup R\cup x_1x_3$, a contradiction. Now we assume that $y_1y_2$ is the chord of $C$. If $P_1$ is odd, then $P_2$ is even. Thus $xP_1x_1y_3$ and $xP_2x_2y_1y_2x_3y_3$ are two even $(x,y)$-paths. Together with an even $(x,y)$-path in $R$, we find an $\varTheta^e$, a contradiction. If $P_1$ is even, then $P_2$ is odd. Thus $P_1$ and $P_2x_2y_2y_1x_1$ are two even $(x,x_1)$-paths. Together with an odd $(x,y)$-path in $R$ and $y_3x_1$, we find an $\varTheta^e$, again a contradiction.
    \end{proof}

    Now by Claim \ref{ClOneChordOneNeighbor}, and by the choice of $\{x,y\}$, we have that $\{x,y\}=N_C(B_1)$, say $N_X(B_1)=\{x\}$ and $N_Y(B_1)=\{y\}$. We assume without loss of generality that $y_1y_3$ is the chord of $C$. It follows that $x=x_1$; for otherwise $y_1y_3$ is a good cut of $G$. We also have $y=y_2$; for otherwise $C$ contains two adjacent vertices of degree 2. Henceforth $G_1$ has the construction $T_2$, as desired.
  \end{proof}

  \begin{claim}\label{ClNoGoodCut}
    $G$ has no good cut.
  \end{claim}

  \begin{proof}
    Suppose that $\{x_0,y_0\}$ is a good cut of $G$ and $B_0,D_0$ be the two components of $G-\{x_0,y_0\}$. Let $\{x_1,y_1\}$ be a good cut with $x_1,y_1\notin V(B_0)$ such that the component of $G-\{x_1,y_1\}$ containing $B_0$ is as large as possible, and let $\{x_2,y_2\}$ be a good cut with $x_2,y_2\notin V(D_0)$ such that the component of $G-\{x_2,y_2\}$ containing $D_0$ is as large as possible (possibly $\{x_1,y_1\}\cap\{x_2,y_2\}\neq\emptyset$). Let $H_1$ be the component of $G-\{x_1,y_1\}$ not containing $B_0$, and $H_2$ be the component of $G-\{x_2,y_2\}$ not containing $D_0$. By Claim \ref{ClGoodCut}, $G_i:=G[V(H_i)\cup\{x_i,y_i\}]$ has the construction $T_1(x_i,y_i)$ or $T_2(x_i,y_i)$, $i=1,2$.

    Since $G$ is 2-connected, there are two vertex-disjoint paths from $\{x_1,y_1\}$ to $\{x_2,y_2\}$. We let $P^x,P^y$ be such two paths with $|P^x|+|P^y|$ as small as possible (specially, $P^x$ and $P^y$ are induced paths). We assume without loss of generality that $\en(P^x)=\{x_1,x_2\}$ and $\en(P^y)=\{y_1,y_2\}$.

    \begin{subclaim}\label{ClEquiv}
      If $P_1,P_2$ are two vertex-disjoint paths from $\{x_1,y_1\}$ to $\{x_2,y_2\}$, then $|P_1|+|P_2|\equiv|P^x|+|P^y|\bmod 4$.
    \end{subclaim}

    \begin{proof}
      Notice that $T_1(x,y)$ has an $(x,y)$-path of length 1 and an $(x,y)$-path of length 2, $T_2(x,y)$ has an $(x,y)$-path of length 3 and an $(x,y)$-path of length 4. If $G_1$ and $G_2$ have both construction $T_1$ or have both construction $T_2$, then $|P^x|+|P^y|\equiv|P_1|+|P_2|\equiv 3\bmod 4$; if one of $G_1,G_2$ has construction $T_1$ the other has construction $T_2$, then $|P^x|+|P^y|\equiv|P_1|+|P_2|\equiv 1\bmod 4$.
    \end{proof}

    \begin{subclaim}\label{ClNoComponent}
      $V(G)=V(H_1)\cup V(H_2)\cup V(P^x)\cup V(P^y)$.
    \end{subclaim}

    \begin{proof}
      Let $H$ be a component of $G-H_1-H_2-P^x-P^y$. We claim that there is an even path between two vertices in $P^x\cup P^y$ and with all internal vertices in $H$. Suppose first that $H$ has at least three neighbors in $P^x\cup P^y$, say $u_1,u_2,u_3\in N_{P^x\cup P^y}(H)$. Then there are three internally-disjoint paths $P_1,P_2,P_3$ from $u\in V(H)$ to $u_1,u_2,u_3$, respectively. It follows that either $P_1uP_2$ or $P_1uP_3$ or $P_2uP_3$ is an even path, as desired. Now assume that $H$ has only two neighbors $u_1,u_2\in V(P^x\cup P^y)$. If $H$ is nontrivial, then by Claim \ref{ClOddCycle}, there is an adjustable $(u_1,u_2)$-path in $G[V(H)\cup\{u_1,u_2\}]$, which contains an even path from $u_1$ to $u_2$. If $H$ is trivial, say $V(H)=\{u\}$, then $u_1uu_2$ is an even path from $u_1$ to $u_2$, as desired.

      Now let $P$ be an even path with $\en(P)=\{u_1,u_2\}\subseteq V(P^x)\cup V(P^y)$, with all internal vertices in $H$. Suppose first that $u_1\in V(P^x)$ and $u_2\in V(P^y)$. Notice that $G_1$ contains an adjustable $(x_1,y_1)$-path. Together with $P^x[x_1,u_1]$ and $P^y[y_1,u_2]$, we get an adjustable $(u_1,u_2)$-path, which contains an even $(u_1,u_2)$-path $P_1$ in $G_1\cup P^x[x_1,u_1]\cup P^y[y_1,u_2]$. Similarly there is an even $(u_1,u_2)$-path $P_2$ in $G_2\cup P^x[x_2,u_1]\cup P^y[y_2,u_2]$. It follows that $P\cup P_1\cup P_2$ is a $\varTheta^e$, a contradiction.

      Now assume without loss of generality that both $u_1,u_2\in V(P^x)$, and that $x_1,u_1,u_2,x_2$ appear in this order along $P^x$. Let $P_1=P^x[x_1,u_1]u_1Pu_2P^x[u_2,x_2]$. By Claim \ref{ClEquiv}, $|P_1|+|P^y|\equiv|P^x|+|P^y|\bmod 4$, implying that $|P^x[u_1,u_2]|\equiv|P|\bmod 4$. It follows that $P^x[u_1,u_2]u_2Pu_1$ is a $(0\bmod 4)$-cycle, a contradiction.
    \end{proof}

    \begin{subclaim}\label{ClTwoEdge}
      There are at most two edges between $P^x$ and $P^y$.
    \end{subclaim}

    \begin{proof}
      Here we say two edges $u_1v_1$ and $u_2v_2$ with $u_1,u_2\in V(P^x)$, $v_1,v_2\in V(P^y)$ are \emph{crossed} if $u_1$ appears before $u_2$ in $P^x$ and $v_2$ appears before $v_1$ in $P^y$. We first claim that each two edges between $P^x$ and $P^y$ are not crossed. Suppose otherwise that $u_1v_1$ and $u_2v_2$ are crossed. If $u_1u_2\in E(P^x)$ and $v_1v_2\in E(P^y)$, then $u_1u_2v_2v_1u_1$ is a 4-cycle, a contradiction. So assume that $|P^x[u_1,u_2]|+|P^y[v_1,v_2]|\geq 3$. Let $P_1=P^x[x_1,u_1]u_1v_1P^y[v_1,y_2]$ and $P_2=P^y[y_1,v_2]v_2u_2P^x[u_2,x_2]$. Then $P_1,P_2$ are two vertex-disjoint paths from $\{x_1,y_1\}$ to $\{x_2,y_2\}$ with $|P_1|+|P_2|<|P^x|+|P^y|$, contradicting the choice of $P^x,P^y$.

      Now let $u_1v_1,u_2v_2,u_3v_3$ be three edges between $P^x$ and $P^y$. Since each two of the three edges are not crossed, we can assume that $u_1,u_2,u_3$ appear in this order along $P^x$ and $v_1,v_2,v_3$ appear in this order along $P^y$. We choose $u_1v_1,u_2v_2,u_3v_3$ such that $|P^x[u_1,u_3]|+|P^y[v_1,v_3]|$ is as small as possible, which follows that they are the only edges between $P^x[u_1,u_3]$ and $P^y[v_1,v_3]$.

      Let $C_1=P^x[u_1,u_2]u_2v_2P^y[v_2,v_1]v_1u_1$ and $C_2=P^x[u_2,u_3]u_3v_3P^y[v_3,v_2]v_2u_2$. If both $C_1$ and $C_2$ are triangle, then $P^x[u_1,u_3]u_3v_3P^y[v_3,v_1]v_1u_1$ is a 4-cycle, a contradiction. Thus we assume without loss of generality that $C_1$ is not a triangle, which implies that $|C_1|\geq 5$. Notice that all the vertices in $V(C_1)\backslash\{u_1,u_2,v_1,v_2\}$ have degree 2 in $G$. If $u_1=u_2$, then two adjacent vertices in $P^y(v_1,v_2)$ are of degree 2, a contradiction. Thus we have that $u_1\neq u_2$ and similarly $v_1\neq v_2$. Clearly $\{u_1,v_2\}$ is a cut of $G$. Let $G'$ be the switching of $G$ at $\{u_1,v_2\}$. Then $G'$ has two adjacent vertices of degree 2, a contradiction.
    \end{proof}
    
    By Claims \ref{ClNoComponent} and \ref{ClTwoEdge}, we have that $n=|P^x|+|P^y|+|V(H_1)|+|V(H_2)|+2$, and $e(G)\leq|P^x|+|P^y|+\rho(V(H_1))+\rho(V(H_2))+2$.

    Suppose first that both $G_1$, $G_2$ have construction $T_1$. Then $n=|P^x|+|P^y|+4$, and $e(G)=|P^x|+|P^y|+6$ (notice that in this case $x_1y_1$ and $x_2y_2$ are the two edges between $P^x$ and $P^y$). Recall that $|P^x|+|P^y|\equiv 3\bmod 4$. It follows that $n\geq 7$ and $e(G)=n+2\leq\lfloor\frac{19}{12}(n-1)\rfloor$.

    Suppose second that $G_1$ has construction $T_1$ and $G_2$ has construction $T_2$. Then $n=|P^x|+|P^y|+7$, and $e(G)\leq|P^x|+|P^y|+11$. Recall that $|P^x|+|P^y|\equiv 1\bmod 4$. If $|P^x|+|P^y|=1$, then either $x_1=x_2,P^y=y_1y_2$ or $P^x=x_1x_2,y_1=y_2$. Since $x_1y_1\in E(G)$ and $x_2y_2\notin E(G)$, there is only one edges between $P^x$ and $P^y$. Thus $n=8$ and $e(G)=11$, as desired. If $|P^x|+|P^y|\geq 5$, then $n\geq 12$ and $e(G)\leq n+4\leq\lfloor\frac{19}{12}(n-1)\rfloor$.

    Suppose third that both $G_1$, $G_2$ have construction $T_2$. Then $n=|P^x|+|P^y|+10$, and $e(G)\leq|P^x|+|P^y|+16$. Recall that $|P^x|+|P^y|\equiv 3\bmod 4$. It follows that $n\geq 13$ and $e(G)\leq n+6\leq\lfloor\frac{19}{12}(n-1)\rfloor$.
  \end{proof}

  By Claim \ref{ClNoGoodCut}, we see that if $x$ is a vertex of degree 2 in $G$, then its two neighbors are nonadjacent. Since $G$ is 2-connected and planar, every face of $G$ is (bounded by) a cycle. By a \emph{3-path} we mean a path of order 3.

  \begin{claim}\label{ClFaceJoint}
    Suppose $C_1,C_2$ are two faces of $G$. If $C_1$ and $C_2$ are joint, then they intersect at a vertex, or an edge, or a 3-path.
  \end{claim}

  \begin{proof}
    We first remark that every face of $G$ has no chord: If $C$ is a face with a chord $u_1u_2$. Then $\{u_1,u_2\}$ is a good cut of $G$, contradicting Claim \ref{ClNoGoodCut}.

    Suppose that $u_1,u_2\in V(C_1)\cap V(C_2)$ with $u_1u_2\notin E(C_1)$. Then, $u_1u_2\notin E(G)$. This implies that $\{u_1,u_2\}$ is a cut of $G$, which is not a good cut by Claim \ref{ClNoGoodCut}. Let $u$ be the vertex in the trivial component of $G-\{u_1,u_2\}$. It follows that $u_1uu_2$ is a 3-path in both $C_1,C_2$. If $V(C_1)\cap V(C_2)=\{u_1,u,u_2\}$, then $C_1,C_2$ intersect at the 3-path. Suppose now that there is a forth vertex $v\in V(C_1)\cap V(C_2)$. Then $uv\notin E(G)$. By the analysis above we see that $uu_1v$ or $uu_2v$ is a 3-path in both $C_1,C_2$. Now $u,u_1$ or $u,u_2$ are two adjacent vertices of degree 2, a contradiction.
  \end{proof}

  \begin{claim}\label{ClOneTrangle}
    $G$ has at most one triangle.
  \end{claim}

  \begin{proof}
    Let $C_1$, $C_2$ be two triangles of $G$. If $C_1,C_2$ intersect at an edge, then $C_1\cup C_2$ contains a 4-cycle, a contradiction. If $C_1,C_2$ are vertex-disjoint, then by Claim \ref{ClNoGoodCut} there are three vertex-disjoint paths $P_1,P_2,P_3$ from $C_1$ to $C_2$. By Lemma \ref{LeTwoCycleBrige}, $C_1\cup C_2\cup P_1\cup P_2\cup P_3$ contains a $(0\bmod 4)$-cycle, a contradiction. Now assume that $C_1$ and $C_2$ intersect at a vertex $x$.

    Recall that $G$ has no good cut. There are two vertex-disjoint paths $P_1,P_2$ from $C_1-x$ to $C_2-x$ in $G-x$. Set $C_1=xy_1y_2x$, $C_2=xz_1z_2x$ and $\en(P_i)=\{y_i,z_i\}$, $i=1,2$. If $P_1$ is even, then $C_1\cup C_2\cup P_1$ contains a $(0\bmod 4)$-cycle by Lemma \ref{LeTwoCycleBrige}. If $|P_1|\equiv 1\bmod 4$, then $P_1z_1xy_2y_1$ is a $(0\bmod 4)$-cycle. Now assume that $|P_1|\equiv 3\bmod 4$, and similarly, $|P_2|\equiv 3\bmod 4$. Thus $P_1z_1z_2P_2y_2y_1$ is a $(0\bmod 4)$-cycle, a contradiction.
  \end{proof}

  \begin{claim}\label{ClTwo5Face}
    $G$ has at most five 5-faces.
  \end{claim}

  \begin{proof}
    If there are two 5-faces $C_1,C_2$ that intersect at an edge, then $C_1\cup C_2$ contains an 8-cycle, a contradiction. If two 5-faces $C_1,C_2$ are vertex-disjoint, then by Claim \ref{ClNoGoodCut}, there are three vertex-disjoint paths $P_1,P_2,P_3$ from $C_1$ to $C_2$. By Lemma \ref{LeTwoCycleBrige}, $C_1\cup C_2\cup P_1\cup P_2\cup P_3$ contains a $(0\bmod 4)$-cycle, a contradiction. Thus we conclude that each two 5-faces intersect at a vertex or a 3-path by Claim \ref{ClFaceJoint}.

    \begin{subclaim}\label{ClCommonP3}
      There are no three 5-faces that pairwise intersect at a 3-path.
    \end{subclaim}

    \begin{proof}
      Suppose that $C_1,C_2,C_3$ are three 5-faces that pairwise intersect at a 3-path. Let $C_1,C_2$ intersect at $x_1y_1z_1$. It follows that $d(y_1)=2$ and $y_1\notin V(C_3)$. Since $C_2,C_3$ also intersect at a 3-path, we have that either $x_1$ or $z_1\in V(C_3)$ (but not both). Without loss of generality assume that $x_1\in V(C_3)$ and that $C_2,C_3$ intersect at $x_1y_2z_2$. Thus $d(y_2)=2$ and $x_1\in V(C_1)\cap V(C_3)$. This implies that $C_1,C_3$ intersect at a 3-path starting from $x_1$, say $x_1y_3z_3$. It follows that $d(x_1)=3$ and $d(y_1)=d(y_2)=d(y_3)=2$. Set $U=\{x_1,y_1,y_2,y_3\}$. We have that $\rho(U)=6$ with $|U|=4$, contradicting Claim \ref{ClSubsetU}.
    \end{proof}

    \begin{subclaim}\label{ClCommonVertex}
      There are no three 5-faces that pairwise intersect at a vertex.
    \end{subclaim}

    \begin{proof}
      Suppose that $C_1,C_2,C_3$ are three 5-faces that pairwise intersect at a vertex. Suppose first that $V(C_1)\cap V(C_2)\cap V(C_3)=\emptyset$. Let $C_i,C_{i+1}$ intersect at $x_i$, $i=1,2,3$ (the subscripts are taken modular 3). Then $C_i$ is an adjustable $(x_{i-1},x_i)$-path. It follows that $C_1\cup C_2\cup C_3$ is an $N^o$, a contradiction. Now suppose that $V(C_1)\cap V(C_2)\cap V(C_3)=\{x\}$.

      If there is a component $H$ of $G-C_1-C_2-C_3$ such that $H$ has neighbors in $C_i-x$ for all $i=1,2,3$, then there are three pairwise internally-disjoint paths $P_1,P_2,P_3$ from $y\in V(H)$ to $C_1-x,C_2-x,C_3-x$, respectively. By Lemma \ref{LeThreeCyclePath}, $C_1\cup C_2\cup C_3\cup P_1\cup P_2\cup P_3$ contains a $(0\bmod 4)$-cycle, a contradiction. Now we assume that there are no component of $G-C_1-C_2-C_3$ that has neighbors in all $C_i-x$, $i=1,2,3$.

      We will show that there is a path from $C_1-x$ to $C_2-x$ in $G-C_3$. Recall that $C_1,C_2,C_3$ are three faces of $G$ with a common vertex $x$. We suppose that $C_1,C_2,C_3$ are distributed around $x$ counterclockwise, and we give orientations of $C_1,C_2,C_3$ counterclockwise. Suppose that there are no bridges from $C_1-x$ to $C_2-x$ in $G-C_3$. It follows that for every component $H$ of $G-C_1-C_2-C_3$, either $N(H)\subseteq V(C_1)\cup V(C_3)$ or $N(H)\subseteq V(C_2)\cup V(C_3)$. By our distribution of $C_1,C_2,C_3$, there is a vertex $y\in V(C_3)\backslash\{x\}$ such that for every component $H$ of $G-C_1-C_2-C_3$, either $N(H)\subseteq V(C_1)\cup V(C_3[x,y])$ or $N(H)\subseteq V(C_2)\cup V(C_3[y,x])$. It follows that $\{x,y\}$ is a good cut of $G$, contradicting Claim \ref{ClNoGoodCut}.

      Now we conclude that there is a path $P_1$ from $C_1-x$ to $C_2-x$ in $G-C_3$. By the similar analysis, there is a path $P_2$ from $C_2-x$ to $C_3-x$ in $G-C_1$, and there is a path $P_3$ from $C_3-x$ to $C_1-x$ in $G-C_2$. By Lemma \ref{LeThreeCycleBridge}, $C_1\cup C_2\cup C_3\cup P_1\cup P_2\cup P_3$ contains a $(0\bmod 4)$-cycle, a contradiction.
    \end{proof}

    Notice that the Ramsey number $r(3,3)=6$. If $G$ has at least six 5-faces, then three of them pairwise intersect at either a vertex or a 3-path, contradicting Claims \ref{ClCommonP3} and \ref{ClCommonVertex}.
  \end{proof}

  Let $f$ be the number of faces of $G$, and $f_i$, $i\geq 3$, be the number of $i$-faces of $G$. By Claims \ref{ClOneTrangle} and \ref{ClTwo5Face}, and that $G$ has no $(0\bmod 4)$-cycle, we have that $f_3\leq 1$, $f_4=0$ and $f_5\leq 5$. By Eular's formula,
  $$n+f=2+e(G)=2+\frac{1}{2}\sum_{i\geq 3}if_i\geq 2+3f-\frac{3}{2}f_3-f_4-\frac{1}{2}f_5.$$
  That is
  $$f\leq\frac{1}{2}\left(n-2+\frac{3}{2}f_3+f_4+\frac{1}{2}f_5\right)\leq\frac{1}{2}(n+2).$$
  Thus $e(G)=n+f-2\leq\frac{3}{2}n-1\leq\frac{19}{12}(n-1)$ (when $n\geq 7$), implying that $e(G)\leq\lfloor\frac{19}{12}(n-1)\rfloor$.

  The proof is complete.

\section{Extremal graphs}\label{extremal}

Define $L_8$ and $L_{13}$ to be the graphs show in Figure 6.

\begin{center}
\begin{picture}(220,90)\label{FiL8L13}
\thicklines \put(20,10){\put(0,0){\circle*{4}} \put(0,40){\circle*{4}} \put(20,20){\circle*{4}} \put(20,40){\circle*{4}} \put(40,0){\circle*{4}} \put(40,40){\circle*{4}} \put(20,70){\circle*{4}} \put(60,20){\circle*{4}}
\put(0,0){\line(1,0){40}} \put(0,0){\line(0,1){40}} \put(20,20){\line(0,1){50}} \put(40,0){\line(0,1){40}} \put(20,20){\line(-1,-1){20}} \put(20,20){\line(1,-1){20}} \put(20,70){\line(-2,-3){20}} \put(20,70){\line(2,-3){20}} \put(40,0){\line(1,1){20}} \put(40,40){\line(1,-1){20}} }

\put(100,10){\put(15,0){\circle*{4}} \put(0,25){\circle*{4}} \put(0,45){\circle*{4}} \put(30,25){\circle*{4}} \put(30,45){\circle*{4}} \put(15,70){\circle*{4}} \put(50,35){\circle*{4}} \put(85,0){\circle*{4}} \put(70,25){\circle*{4}} \put(70,45){\circle*{4}} \put(100,25){\circle*{4}} \put(100,45){\circle*{4}} \put(85,70){\circle*{4}}
\put(15,0){\line(3,5){15}} \put(15,0){\line(-3,5){15}} \put(0,25){\line(0,1){20}} \put(30,25){\line(0,1){20}} \put(15,70){\line(3,-5){15}} \put(15,70){\line(-3,-5){15}} \put(0,45){\line(1,0){30}} \put(15,0){\line(1,0){70}} \put(15,0){\line(1,1){70}} \put(15,70){\line(1,-1){70}}
\put(85,0){\line(3,5){15}} \put(85,0){\line(-3,5){15}} \put(70,25){\line(0,1){20}} \put(100,25){\line(0,1){20}} \put(85,70){\line(3,-5){15}} \put(85,70){\line(-3,-5){15}} \put(70,45){\line(1,0){30}} }
\end{picture}

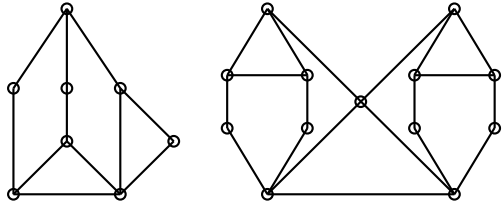
\small Figure 6. The graphs $L_8$ and $L_{13}$.
\end{center}

For $n\geq 2$, we define the graph $G_n$ as follows: Let
\begin{align*}
  n-1   &=12q_1+r_1, 0\leq r_1\leq 11;\\
  r_1   &=7q_2+r_2, 0\leq r_2\leq 6;\\
  r_2   &=2q_3+r_3, 0\leq r_3\leq 1.
\end{align*}
Let $G_n$ be a connected graph consisting of $q_1$ blocks isomorphic to $L_{13}$, $q_2$ blocks isomorphic to $L_8$, $q_3$ blocks isomorphic to $K_3$ and $r_3$ blocks isomorphic to $K_2$. One can compute that $G_n$ contains no $(0\bmod 4)$-cycle and $e(G_n)=\lfloor\frac{19}{12}(n-1)\rfloor$.

Let $\mathcal{C}_{0\bmod 4}$ be the set of all $(0\bmod 4)$-cycles. We have
$$\ex(n,\mathcal{C}_{0\bmod 4})=\left\lfloor\frac{19}{12}(n-1)\right\rfloor.$$

\section{Acknowledgements}
The research of Gy\H{o}ri and Salia was supported by NKFIH, grant K132696.  
The research of Li was supported by NSFC (12071370) and Shaanxi Fundamental Science Research Project for Mathematics and Physics (22JSZ009).
The research of Tompkins was supported by NKFIH, grant K135800. The research of Zhu was supported by HORIZON, grant 101086712.

\bibliographystyle{abbrv}
\bibliography{references.bib}

\begin{thebibliography}{10}

\bibitem{barefoot1991cycles}
C.~Barefoot, L.~Clark, J.~Douthett, R.~Entringer, and M.~Fellows.
\newblock Cycles of length 0 modulo 3 in graphs.
\newblock {\em preprint}, 1991.

\bibitem{bollobas1977cycles}
B.~Bollob{\'a}s.
\newblock Cycles modulo $k$.
\newblock {\em Bulletin of the London Mathematical Society}, 9(1):97--98, 1977.

\bibitem{chen1994graphs}
G.~Chen and A.~Saito.
\newblock Graphs with a cycle of length divisible by three.
\newblock {\em Journal of Combinatorial Theory, Series B}, 60(2):277--292,
  1994.

\bibitem{clapham1989graphs}
C.~Clapham, A.~Flockhart, and J.~Sheehan.
\newblock Graphs without four-cycles.
\newblock {\em Journal of Graph theory}, 13(1):29--47, 1989.

\bibitem{Deanetal}
N.~Dean, A.~Kaneko, K.~Ota, and B.~Toft.
\newblock Cycles modulo 3.
\newblock {\em Dimacs Technical Report}, 91(32), 1991.

\bibitem{dean1993cycles}
N.~Dean, L.~Lesniak, and A.~Saito.
\newblock Cycles of length 0 modulo 4 in graphs.
\newblock {\em Discrete mathematics}, 121(1-3):37--49, 1993.

\bibitem{erdos}
P.~{Erd\H{o}s}.
\newblock Some recent problems and results in graph theory, combinatorics and
  number theory.
\newblock {\em Proceedings of the Seventh Southeastern Conference on
  Combinatorics, Graph Theory, and Computing (Louisiana State Univ., Baton
  Rouge, La., 1976), Congress. Numer. XVII}, pages 3--14, 1976.

\bibitem{erdos1995some}
P.~Erdos.
\newblock Some of my favourite problems in number theory, combinatorics, and
  geometry.
\newblock {\em Resenhas do Instituto de Matem{\'a}tica e Estat{\'\i}stica da
  Universidade de S{\~a}o Paulo}, 2(2):165--186, 1995.

\bibitem{gao2022two}
J.~Gao, B.~Li, J.~Ma, and T.~Xie.
\newblock On two cycles of consecutive even lengths.
\newblock {\em arXiv preprint arXiv:2210.03959}, 2022.

\bibitem{mei2001cycles}
L.~Mei and Y.~Zhengguang.
\newblock Cycles of length 1 modulo 3 in graph.
\newblock {\em Discrete applied mathematics}, 113(2-3):329--336, 2001.

\bibitem{saito1992cycles}
A.~Saito.
\newblock Cycles of length 2 modulo 3 in graphs.
\newblock {\em Discrete mathematics}, 101(1-3):285--289, 1992.

\bibitem{simonovits1974extremal}
M.~Simonovits.
\newblock Extremal graph problems with symmetrical extremal graphs. additional
  chromatic conditions.
\newblock {\em Discrete Mathematics}, 7(3-4):349--376, 1974.

\bibitem{sudakov}
B.~Sudakov and J.~Verstra{\"{e}}te.
\newblock The extremal function for cycles of length {$\ell$}.
\newblock {\em The Electronic Journal of Combinatorics}, 24(1), 2017.

\bibitem{thomassen1983graph}
C.~Thomassen.
\newblock Graph decomposition with applications to subdivisions and path
  systems modulo $k$.
\newblock {\em Journal of Graph Theory}, 7(2):261--271, 1983.

\bibitem{thomassen1986paths}
C.~Thomassen.
\newblock {\em Paths, circuits and subdivisions}.
\newblock Danmarks Tekniske H{\o}jskole. Matematisk Institut, 1986.

\bibitem{verstraete2000arithmetic}
J.~Verstra{\"e}te.
\newblock On arithmetic progressions of cycle lengths in graphs.
\newblock {\em Combinatorics, Probability and Computing}, 9(4):369--373, 2000.

\bibitem{verstraete2016extremal}
J.~Verstra{\"e}te.
\newblock Extremal problems for cycles in graphs.
\newblock In {\em Recent trends in combinatorics}, pages 83--116. Springer,
  2016.

\end{thebibliography}

\end{document}